\documentclass[a4paper,reqno,11pt]{amsart}
\usepackage{amsfonts}
\usepackage{amsmath}
\usepackage{amssymb}
\usepackage{tikz-cd}
\usepackage{cite}
\usepackage{bbm}
\usepackage{amscd}
\usepackage{enumerate}
\usepackage[left=2.8cm,right=2.8cm,bottom=3cm,top=3cm]{geometry}
\usepackage{etoolbox}
\usepackage{filecontents}
\usepackage[hidelinks]{hyperref}

\newtheorem{theorem}{Theorem}[section]

\theoremstyle{definition}
\newtheorem{lemma}[theorem]{Lemma}
\newtheorem{definition}[theorem]{Definition}
\newtheorem{corollary}[theorem]{Corollary}

\begin{document}
\title{Gowers' Ramsey theorem for generalized tetris operations}
\author{Martino Lupini}
\address{Mathematics Department\\
California Institute of Technology\\
1200 E. California Blvd\\
MC 253-37\\
Pasadena, CA 91125}
\email{lupini@caltech.edu}
\urladdr{http://www.lupini.org/}
\date{\today }

\begin{abstract}
We prove a generalization of Gowers' theorem for $\mathrm{FIN}_{k}$ where,
instead of the single tetris operation \ $T:\mathrm{FIN}_{k}\rightarrow 
\mathrm{FIN}_{k-1}$, one considers all maps from $\mathrm{FIN}_{k}$ to $%
\mathrm{FIN}_{j}$ for $0\leq j\leq k$ arising from nondecreasing surjections 
$f:\left\{ 0,1,\ldots ,k+1\right\} \rightarrow \left\{ 0,1,\ldots
,j+1\right\} $. This answers a question of Barto\v{s}ov\'{a} and
Kwiatkowska. We also prove a common generalization of such a result and the
Galvin--Glazer--Hindman theorem on finite products, in the setting of
layered partial semigroups introduced by Farah, Hindman, and McLeod.
\end{abstract}


\maketitle

\section{Introduction\label{Section:introduction}}

Gower's theorem on $\mathrm{FIN}_{k}$ is a generalization of Hindman's
theorem on finite unions where one considers, rather than finite nonempty
subsets of $\omega $, the space $\mathrm{FIN}_{k}$ of all finitely supported
functions from $\omega $ to $\left\{ 0,1,\ldots ,k\right\} $ with maximum
value $k$. Such a space is endowed with a natural operation of pointwise
sum, which is defined for pairs of functions with disjoint support. Gowers
considered also the \emph{tetris operation}\textrm{\ }$T:\mathrm{FIN}%
_{k}\rightarrow \mathrm{FIN}_{k-1}$ defined by letting $\left( Tb\right)
(n)=\max \left\{ b(n)-1,0\right\} $ for $b\in \mathrm{FIN}_{k}$. Gowers'
theorem can be stated, shortly, by saying that for any finite coloring of $%
\mathrm{FIN}_{k}$ there exists an infinite sequence $\left( b_{n}\right) $
which is a\emph{\ block sequence}---in the sense that every element of the
support of $b_{n}$ precedes every element of the support of $b_{n+1}$---with
the property that the intersection of $\mathrm{FIN}_{k}$ with the smallest
subset of $\mathrm{FIN}_{1}\cup \cdots \cup \mathrm{FIN}_{k}$ that contains
the $b_{n}$'s and it is closed under pointwise sum of disjointly supported
functions and under the tetris operation, is monochromatic \cite%
{gowers_lipschitz_1992}. Gowers then used such a result---or more precisely
its symmetrized version where one considers functions from $\omega $ to $%
\left\{ -k,\ldots ,k\right\} $---to prove an \emph{oscillation stability }%
result for the sphere of the Banach space $c_{0}$. Other proof of Gowers'
theorem can be found in \cite%
{gowers_ramsey_2003,kanellopoulos_proof_2004,todorcevic_introduction_2010}.

Gowers' theorem of $\mathrm{FIN}_{k}$ as stated above implies through a
standard compactness argument its corresponding \emph{finitary version}.
Explicit combinatorial proofs of such a finitary version have been recently
given, independently, by Tyros \cite{tyros_primitive_2015} and
Ojeda-Aristizabal \cite{ojeda-aristizabal_finite_2015}. Particularly, the
argument from \cite{tyros_primitive_2015} yields a primitive recursive bound
on the associated \emph{Gowers numbers}.

A broad generalization of Gowers' theorem has been proved by Farah, Hindman,
and McLeod in \cite[Theorem 3.13]{farah_partition_2002} in the framework,
developed therein, of \emph{layered partial semigroups }and \emph{layered
actions}. Such a result provides, in particular, a common generalization of
Gowers' theorem and the Hales--Jewett theorem; see \cite[Theorem 3.15]%
{farah_partition_2002}. As general as \cite[Theorem 3.13]%
{farah_partition_2002} is, it nonetheless does not cover the case where one
considers $\mathrm{FIN}_{k}$ endowed with the multiple tetris operations
described below, since these do not form a layered action in the sense of 
\cite[Definition 3.3]{farah_partition_2002}.

In \cite{bartosova_gowers_2014}, Barto\v{s}ov\'{a} and Kwiatkowska
considered a generalization of Gowers' theorem, where \emph{multiple tetris
operations} are allowed. Precisely, they defined for $1\leq i\leq k$ the
tetris operation $T_{i}:\mathrm{FIN}_{k}\rightarrow \mathrm{FIN}_{k-1}$ by%
\begin{equation*}
T_{i}\left( b\right) :n\mapsto \left\{ 
\begin{array}{ll}
b(n)-1 & \text{if }b(n)\geq i\text{, and} \\ 
b(n) & \text{otherwise.}%
\end{array}%
\right. 
\end{equation*}%
Adapting methods from \cite{tyros_primitive_2015}, Barto\v{s}ov\'{a} and
Kwiatkowska proved in \cite{bartosova_gowers_2014} the strengthening of the 
\emph{finitary version }of Gowers' theorem where multiple tetris operations
are considered. The authors then provided in \cite{bartosova_gowers_2014}
applications of such a result to the dynamics of the Lelek fan.

Question 8.3 of \cite{bartosova_gowers_2014} asks whether the \emph{%
infinitary version }of Gowers' theorem on $\mathrm{FIN}_{k}$ holds when one
considers multiple tetris operations. In this paper, we show that this is
the case, via an adaptation of Gowers' original argument using idempotent
ultrafilters. In order to precisely state our result, we introduce some
terminology, to be used in the rest of the paper.

We denote by $\omega $ the set of nonnegative integers, and by $\mathbb{N}$
the set of nonzero elements of $\omega $. We identify an element $k$ of $%
\omega $ with the set $\left\{ 0,1,\ldots ,k-1\right\} $ of its
predecessors. As mentioned above, $\mathrm{FIN}_{k}$ denotes the set of
functions from $\omega $ to $k+1$ with maximum value $k$ and that vanish for
all but finitely many elements of $\omega $. We also let $\mathrm{FIN}_{\leq
k}$ be the union of $\mathrm{FIN}_{j}$ for $j=1,2,\ldots ,k$. The \emph{%
support }$\mathrm{Supp}\left( b\right) $ of an element $b$ of $\mathrm{FIN}%
_{k}$ is the set of elements of $\omega $ where $b$ does \emph{not }vanish.
For finite nonempty subsets $F,F^{\prime }$ of $\omega $, we write $%
F<F^{\prime }$ if the maximum element of $F\ $is smaller than the minimum
element of $F^{\prime }$.

Suppose that $0\leq j\leq k$ and $f:k+1\rightarrow j+1$ is a nondecreasing
surjection. We also denote by $f$ the \emph{generalized tetris operation }$f:%
\mathrm{FIN}_{k}\rightarrow \mathrm{FIN}_{j}$ defined by $f\left( b\right)
=f\circ b$. It is clear that the class of generalized tetris operations is
precisely the set of mappings that can be obtained as composition of the
multiple tetris operations $T_{i}$ for $i=1,2,\ldots ,k$ \footnote{%
We thank S\l awomir Solecki for pointing this out.}.

We say that $\left( b_{n}\right) $ is a \emph{block sequence }in $\mathrm{FIN%
}_{k}$ if $b_{n}\in \mathrm{FIN}_{k}$ and $\mathrm{Supp}\left( b_{n}\right) <%
\mathrm{Supp}\left( b_{n+1}\right) $ for every $n\in \omega $. If $j\in 
\mathbb{N}$, then we define the \emph{tetris subspace }$\mathrm{TS}%
_{j}\left( b_{n}\right) $ of $\mathrm{FIN}_{j}$ generated by $\left(
b_{n}\right) $ to be the set of elements of $\mathrm{FIN}_{j}$ of the form%
\begin{equation*}
f_{0}\circ b_{0}+\cdots +f_{n}\circ b_{n}
\end{equation*}%
for some $n\in \omega $, $j_{0},\ldots ,j_{n}\in j+1$ such that $\max
\left\{ j_{0},\ldots ,j_{n}\right\} =j$, and nondecreasing surjections $%
f_{i}:k+1\rightarrow j_{i}+1$ for $i\in n$. A block sequence $\left(
b_{n}^{\prime }\right) $ in $\mathrm{FIN}_{k}$ is a \emph{block subsequence}
of $\left( b_{n}\right) $ if $\left( b_{n}^{\prime }\right) $ is contained
in $\mathrm{\mathrm{TS}}_{k}\left( b_{n}\right) $.

In the following we will use some standard terminology concerning colorings.
An $r$-\emph{coloring} (or coloring with $r$ colors) of a set $X$ is a
function $c:X\rightarrow r$, and a \emph{finite coloring} is an $r$-coloring
for some $r\in \omega $. A subset $A$ of $X$ is monochromatic (for the given
coloring $c$) if $c$ is constant on $A$. Using this terminology, we can
state our infinitary Gowers' theorem for generalized tetris operations as
follows.

\begin{theorem}
\label{Theorem:multiGowers}Suppose that $k\in \mathbb{N}$. For any finite
coloring of $\mathrm{FIN}_{\leq k}$, there exists an infinite block sequence 
$\left( b_{n}\right) $ in $\mathrm{FIN}_{k}$ such that $\mathrm{TS}%
_{j}\left( b_{n}\right) $ is monochromatic for every $j=1,2,\ldots ,k$.
\end{theorem}

Theorem \ref{Theorem:multidimGowers} implies via a standard compactness
argument its corresponding finitary version. If $k,n\in \omega $, then we
denote by $\mathrm{FIN}_{k}(n)$ the set of functions $f:n\rightarrow k+1$
with maximum value $k$, and by $\mathrm{FIN}_{\leq k}(n)$ the union of $%
\mathrm{FIN}_{j}(n)$ for $j=1,2,\ldots ,k$. The notion of block sequence $%
\left( b_{0},\ldots ,b_{m-1}\right) $ and tetris subspace $\mathrm{TS}%
_{j}\left( b_{0},\ldots ,b_{m-1}\right) $ of $\mathrm{TS}_{j}(n)$ generated
by $\left( b_{0},\ldots ,b_{m-1}\right) $ are defined similarly as their
infinite counterparts.

\begin{corollary}
\label{Corollary:multiGowers}Given $k,r,\ell \in \mathbb{N}$ there exists $%
n\in \mathbb{N}$ such that for any $r$-coloring of $\mathrm{FIN}_{\leq k}(n)$%
, there exists a block sequence $\left( b_{0},b_{1},\ldots ,b_{\ell
-1}\right) $ in $\mathrm{FIN}_{k}(n)$ of length $\ell $ such that $\mathrm{TS%
}_{j}\left( b_{0},\ldots ,b_{\ell -1}\right) $ is monochromatic for any $%
j=1,2,\ldots ,k$.
\end{corollary}

We will also prove below a more general statement than Theorem \ref%
{Theorem:multiGowers}, where one considers colorings of the space $\mathrm{%
FIN}_{k}^{\left[ m\right] }$ of block sequences of $\mathrm{FIN}_{k}$ of a
fixed length $m$. We will also provide a common generalization of such a
result and the Galvin--Glazer--Hindman theorem on finite products, in the
setting of layered partial semigroups introduced by Farah, Hindman, and
McLeod in \cite{farah_partition_2002}.

As mentioned above, the original Gowers theorem from \cite%
{gowers_lipschitz_1992} was used to prove the following \emph{%
oscillation-stability }result for the positive part of the sphere of $c_{0}$.
Recall that $c_{0}$ denotes the real Banach space of vanishing sequences of
real numbers endowed with the supremum norm. Let $\mathrm{PS}(c_{0})$ be the
positive part of the sphere of $c_{0}$, which is the set of elements of $%
c_{0}$ of norm $1$ with nonnegative coordinates. The support $\mathrm{Supp}%
\left( f\right) $ of an element $f$ of $c_{0}$ is the set $n\in \omega $
such that $f(n)\neq 0$. A \emph{normalized positive block basis }is a
sequence $\left( f_{n}\right) $ of finitely-supported elements of $\mathrm{PS%
}(c_{0})$ such that $\mathrm{Supp}\left( f_{n}\right) <\mathrm{Supp}\left(
f_{n+1}\right) $ for every $n\in \omega $. Gower's oscillation-stability
result asserts that for any Lipschitz map $F:\mathrm{PS}(c_{0})\rightarrow 
\mathbb{R}$ and $\varepsilon >0$ there exists a block basis $\left(
f_{n}\right) $ such that the oscillation of $F$ on the positive part of the
sphere of the subspace of $c_{0}$ spanned by $\left( f_{n}\right) $ is at
most $\varepsilon $ \cite[Theorem 6]{gowers_lipschitz_1992}. Such a result
is proved by considering a suitable discretization of $\mathrm{PS}(c_{0})$
that can naturally be identified with $\mathrm{FIN}_{k}$; see the proof of 
\cite[Theorem 6]{gowers_lipschitz_1992} and also \cite[Corollary 2.26]%
{todorcevic_introduction_2010}. Under such an identification, the tetris
operation on $\mathrm{FIN}_{k}$ corresponds to multiplication by positive
scalars in $c_{0}$.

Similarly, one can observe that the multiple tetris operations $T_{i}$ for $%
i=1,2,\ldots ,k$ described above correspond to the following nonlinear
operators on $c_{0}$. Fix $\lambda ,t\in \left[ 0,1\right] $ and consider
the operator $S_{t,\lambda }$ on $c_{0}$ mapping $f$ to the function 
\begin{equation*}
n\mapsto \left\{ 
\begin{array}{ll}
\lambda f(n) & \text{if }\left\vert f(n)\right\vert \geq t\text{,} \\ 
f(n) & \text{otherwise.}%
\end{array}%
\right. 
\end{equation*}%
Given a normalized positive block basis $\left( f_{n}\right) $, one can
consider the smallest subspace of $c_{0}$ that contains $\left( f_{n}\right) 
$ and it is invariant under $S_{t,\lambda }$ for every $t,\lambda \in \left[
0,1\right] $. Then arguing as in the proof of Gowers' oscillation-stability
theorem one can deduce from Theorem \ref{Theorem:multiGowers} the following
result.

\begin{theorem}
\label{Theorem:oscillation}Suppose that $F:\mathrm{PS}(c_{0})\rightarrow 
\mathbb{R}$ is a Lipschitz map, and $\varepsilon >0$. There exists a
positive normalized block sequence $\left( f_{n}\right) $ such that the
oscillation of $F$ on the positive part of the sphere of the smallest
subspace of $c_{0}$ containing $\left( f_{n}\right) $ and invariant under $%
S_{t,\lambda }$ for $t,\lambda \in \left[ 0,1\right] $ is at most $%
\varepsilon $.
\end{theorem}

The rest of this paper consists of three sections. In Section \ref%
{Section:tetris} we present a proof of Theorem \ref{Theorem:multiGowers}. In
Section \ref{Section:multidimensional} we explain how the proof of Theorem %
\ref{Theorem:multiGowers} can be modified to prove its multidimensional
generalization. Finally in Section \ref{Section:abstract} we recall the
theory of layered partial semigroups developed in \cite{farah_partition_2002}%
, and present in this setting a common generalization of the
(multidimensional version of) Theorem \ref{Theorem:multiGowers} and the
Galvin--Glazer--Hindman theorem on finite products.

\subsection*{Acknowledgments}

We are grateful to David Fernandez, Aleksandra Kwiatkowska, S\l awomir
Solecki, and Kostas Tyros for their comment and suggestions. We are also
thank Aleksandra Kwiatkowska for pointing out a mistake in an earlier
version of this paper, and Ilijas Farah for referring us to \cite%
{farah_partition_2002} and to the theory of layered partial semigroups.

\section{Gowers' theorem for generalized tetris operations\label%
{Section:tetris}}

Our proof of Theorem \ref{Theorem:multiGowers} uses the tool of idempotent
ultrafilters, similarly as Gowers' original proof from \cite%
{gowers_lipschitz_1992}. In the following we will frequently use the
notation of ultrafilter quantifiers \cite[\S 1.1]%
{todorcevic_introduction_2010}, which are defined as follows. If $\mathcal{U}
$ is an ultrafilter on $\mathrm{FIN}_{k}$ and $\varphi (x)$ is a first-order
formula, then $\left( \mathcal{U}b\right) \varphi \left( b\right) $ means
that the set of $b\in \mathrm{FIN}_{k}$ such that $\varphi \left( b\right) $
holds belongs to $\mathcal{U}$. A similar notation applies to ultrafilters
on an arbitrary set.

Adopting the terminology of \cite[\S 2.4]{todorcevic_introduction_2010}, we
say that an ultrafilter $\mathcal{U}$ on $\mathrm{FIN}_{k}$ is \emph{cofinite%
} if $\left( \forall n\in \omega \right) \left( \mathcal{U}b\right) $, $%
b(n)=0$. The set $\gamma \mathrm{FIN}_{k}$ of cofinite ultrafilters on $%
\mathrm{FIN}_{k}$ is endowed with a canonical\emph{\ semigroup operation},
defined by setting $A\in \mathcal{U}+\mathcal{V}$ if and only if $\left( 
\mathcal{U}b\right) \left( \mathcal{V}b^{\prime }\right) $, $\mathrm{Supp}%
\left( b\right) <\mathrm{Supp}\left( b^{\prime }\right) $ and $b+b^{\prime
}\in A$. Furthermore $\gamma \mathrm{FIN}_{k}$ is endowed with a canonical
compact Hausdorff topology. Such a topology has a clopen basis consisting of
sets of the form $\widehat{A}=\left\{ \mathcal{U}\in \gamma \mathrm{FIN}%
_{k}:A\in \mathcal{U}\right\} $ for $A\subset \mathrm{FIN}_{k}$. Endowed
with such a topology and semigroup operation, $\gamma \mathrm{\mathrm{FIN}}%
_{k}$ is a\emph{\ compact right topological semigroup} \cite[\S 2.4]%
{todorcevic_introduction_2010}. This means that the right multiplication map 
$\mathcal{U}\mapsto \mathcal{U}+\mathcal{V}$ is continuous for any $\mathcal{%
V}\in \gamma \mathrm{FIN}_{k}$. Any generalized tetris operation $f:\mathrm{%
FIN}_{k}\rightarrow \mathrm{FIN}_{j}$ admits a canonical extension to a
continuous homomorphism $f:\gamma \mathrm{FIN}_{k}\rightarrow \mathrm{FIN}%
_{j}$, defined by letting $A\in f\left( \mathcal{U}\right) $ iff $\left( 
\mathcal{U}b\right) $ $f\left( b\right) \in A$. In the following we will
repeatedly use the well know result, due to Ellis and Namakura, that any
compact right topological semigroup contains an \emph{idempotent element} 
\cite[Lemma 2.1]{todorcevic_introduction_2010}.

The following lemma can be seen as a refinement of \cite[Lemma 3]%
{gowers_lipschitz_1992}, and it is the core of the proof of Theorem \ref%
{Theorem:multiGowers}.

\begin{lemma}
\label{Lemma:idempotent}There exists a sequence $\left( \mathcal{U}%
_{k}\right) $ of cofinite ultrafilters $\mathcal{U}_{k}$ on $\mathrm{FIN}%
_{k} $ such that for any $0<j\leq k$ and for any nondecreasing surjection $%
f:k+1\rightarrow j+1$, $\mathcal{U}_{k}+\mathcal{U}_{j}=\mathcal{U}_{j}+%
\mathcal{U}_{k}=\mathcal{U}_{k}$ and $f\left( \mathcal{U}_{k}\right) =%
\mathcal{U}_{j}$.
\end{lemma}

\begin{proof}
We define by recursion on $k$ sequences $(p_{j}^{(k)})$ of idempotent
ultrafilters $p_{j}^{(k)}\in \gamma \mathrm{FIN}_{j}$ such that, for every $%
k,i,j\in \mathbb{N}$ and nondecreasing surjection $f:j+1\rightarrow i+1$,

\begin{enumerate}
\item $f(p_{j}^{(k)})=p_{i}^{(k)}$,

\item $p_{j}^{(k+1)}=p_{j}^{(k)}$ for $j\leq k$, and

\item $p_{j}^{(k)}+p_{j-1}^{(k)}=p_{j}^{(k)}$ for $2\leq j\leq k$.
\end{enumerate}

Granted the construction of the sequences $(p_{j}^{(k)})$, we can set $%
\mathcal{U}_{k}:=p_{1}^{(k)}+p_{2}^{(k)}+\cdots +p_{k}^{(k)}$. Observe that
(3) implies that $p_{k}^{(k)}+p_{j}^{(k)}=p_{k}^{(k)}$ for $j\leq k$. Hence $%
\mathcal{U}_{k}$ is idempotent, and $\mathcal{U}_{k}+\mathcal{U}_{k+1}=%
\mathcal{U}_{k+1}+\mathcal{U}_{k}=\mathcal{U}_{k+1}$ for every $k\in \mathbb{%
N}$ by\ (2). Furthermore it follows from (1) and (3) that $f\left( \mathcal{U%
}_{k}\right) =\mathcal{U}_{j}$ for any nondecreasing surjection $%
f:k+1\rightarrow j+1$. Indeed suppose that $f:k+1\rightarrow j+1$ is a
nondecreasing surjection. Then $f|_{i+1}:i+1\rightarrow f\left( i\right) +1$
is a nondecreasing surjection for every $i\in k+1$. Hence we can conclude by
(1),(3), and the fact that the $p_{i}^{(k)}$'s are idempotent that%
\begin{equation*}
f\left( \mathcal{U}_{k}\right) =f(p_{1}^{(k)})+\cdots
+f(p_{k}^{(k)})=p_{f\left( 1\right) }^{(k)}+p_{f\left( 2\right)
}^{(k)}+\cdots +p_{f\left( k\right) }^{(k)}=p_{1}^{(j)}+\cdots +p_{j}^{(j)}=%
\mathcal{U}_{j}\text{.}
\end{equation*}

We now show how to construct the sequences $(p_{j}^{(k)})$. In the following
we will convene that $p_{0}^{(k)}$ is the function on $\omega $ constantly
equal to $0$, which can be seen as the unique element of $\mathrm{FIN}_{0}$.
We let $\Pi \mathrm{\ }$be the product of $\gamma \mathrm{FIN}_{j}$ for $%
j\in \mathbb{N}$. Observe that $\Pi $ has a natural compact right
topological semigroup structure, where the topology is the product topology
and the operation is the entrywise sum.

For $k=1$, consider the compact semigroup $\Sigma _{1}\subset \Pi $ of
sequences $(q_{j})$ satisfying $f\left( q_{j}\right) =q_{i}$ for any $i,j\in 
\mathbb{N}$ and nondecreasing surjection $f:j+1\rightarrow i+1$. We observe
that $\Sigma _{1}$ is nonempty. Define for $j\in \mathbb{N}$ the set $%
M_{j}:=\left\{ b\in \mathrm{FIN}_{j}:\forall n\in \omega \text{, }b(n)\in
\left\{ 0,j\right\} \right\} $. Fix for any $j\geq 2$ a nondecreasing
surjection $f_{j}:j+1\rightarrow j$. Observe that $f_{j}$ maps $M_{j}$
bijectively onto $M_{j-1}$. Furthermore, for any $0<i\leq j$ and
nondecreasing surjection $f:j+1\rightarrow i+1$, one has that $%
f|_{M_{j}}=\left( f_{i+1}\circ f_{i}\circ \cdots \circ f_{j}\right) |_{M_{j}}
$. We denote by $\left( f_{j}|_{M_{j}}\right) ^{-1}:M_{j-1}\rightarrow M_{j}$
the inverse map of $f_{j}$. Let $p_{1}$ be any element of $\gamma \mathrm{FIN%
}_{1}$. One can define a sequence $(q_{j})$ that belongs to $\Sigma _{1}$
and such that $M_{j}\in q_{j}$ by recursion on $j\in \mathbb{N}$, by letting 
$q_{j+1}=\left( f_{j+1}|_{M_{j+1}}\right) ^{-1}(q_{j})$ . This concludes the
proof that $\Sigma _{1}$ is nonempty. We can then let $(p_{j}^{(1)})$ be an
idempotent element of $\Sigma _{1}$. This concludes the construction for $k=1
$.

Suppose that the sequences $(p_{j}^{(\ell )})$ have been defined for $\ell
=1,2,\ldots ,k-1$ satisfying (1),(2),(3) above. We explain how to define $%
(p_{j}^{(k)})$. Consider the compact semigroup $\Sigma _{k}\subset \Pi $ of
sequences $\left( q_{j}\right) $ such that, for any $i,j\in \mathbb{N}$ and
nondecreasing surjection $f:j+1\rightarrow i+1$, $f\left( q_{j}\right)
=q_{i} $, $q_{j}=p_{j}^{(k-1)}$ for $j\in k$, and $q_{j}+p_{i}^{(k-1)}=q_{j}$
for $j\in \mathbb{N}$ and $i\leq \min \left\{ k-1,j\right\} $. We need to
show that $\Sigma _{k}$ is nonempty. Set 
\begin{equation*}
q_{j}:=p_{j}^{(k-1)}+p_{j-1}^{(k-1)}+\cdots +p_{1}^{(k-1)}
\end{equation*}%
for every $j\in \mathbb{N}$. Observe that for $j\leq k-1$ one has that $%
q_{j}=p_{j}^{(k-1)}$ in view of (3). If $j\geq k$ then 
\begin{equation*}
q_{j}=p_{j}^{(k-1)}+p_{j-1}^{(k-1)}+\cdots +p_{k-1}^{(k-1)}\text{,}
\end{equation*}%
again by (3). Since $p_{i}^{(k-1)}$ is idempotent for any $i\in \mathbb{N}$,
it follows from (1) that $f\left( q_{j}\right) =q_{i}$ for any nondecreasing
surjection $f:j+1\rightarrow i+1$. Furthermore for $j\in \mathbb{N}$ and $%
i\leq \min \left\{ k-1,j\right\} $ one has that $q_{j}+p_{i}^{(k-1)}=q_{j}$
in view of (2). This shows that the sequence $\left( q_{j}\right) $ belongs
to $\Sigma _{k}$. One can then let $(p_{j}^{(k)})$ be any idempotent element
of $\Sigma _{k}$. This concludes the recursive construction.
\end{proof}

Theorem \ref{Theorem:multiGowers} can now be deduced from Lemma \ref%
{Lemma:idempotent} through a standard argument. We present a sketch of the
proof, for convenience of the reader.

\begin{proof}[Proof of Theorem \protect\ref{Theorem:multiGowers}]
Suppose that $\mathcal{U}_{1},\mathcal{U}_{2},\ldots ,\mathcal{U}_{k}$ are
the cofinite ultrafilters constructed in Lemma \ref{Lemma:idempotent}. Fix a
finite coloring $c$ of $\mathrm{FIN}_{j}$ for $1\leq j\leq k$, and let $%
A_{j} $ be an element of $\mathcal{U}_{j}$ such that $c$ is constant on $%
A_{j}$ for $j=1,2,\ldots ,k$. We define, by recursion on $n\in \omega $, $%
b_{n}\in \mathrm{FIN}_{k}$ with $\mathrm{Supp}\left( b_{i}\right) <\mathrm{%
Supp}\left( b_{j}\right) $ for $i<j$, such that the following conditions
hold: for any $j_{0},\ldots ,j_{n+2}\in k+1$ and nondecreasing surjections $%
f_{i}:k+1\rightarrow j_{i}+1$ for $i\in n+3$,

\begin{enumerate}
\item $f_{0}\circ b_{0}+\cdots +f_{n}\circ b_{n}$ belongs to $A_{\max
\left\{ j_{0},\ldots ,j_{n}\right\} }$,

\item $\left( \mathcal{U}_{k}y\right) $, $f_{0}\circ b_{0}+\cdots
+f_{n}\circ b_{n}+f_{n+1}\circ y$ belongs to $A_{\max \left\{ j_{0},\ldots
,j_{n+1}\right\} }$, and

\item $\left( \mathcal{U}_{k}y\right) \left( \mathcal{U}_{k}z\right) $, $%
f_{0}\circ b_{0}+\cdots +f_{n}\circ b_{n}+f_{n+1}\circ y+f_{n+2}\circ z$
belongs to $A_{\max \left\{ j_{0},\ldots ,j_{n+2}\right\} }$.
\end{enumerate}

Suppose that such a sequence has been defined up to $n$. From (2) and (3),
we can conclude that there exists $b_{n+1}\in \mathrm{FIN}_{k}$ such that $%
\mathrm{Supp}\left( b_{n+1}\right) >\mathrm{Supp}\left( b_{n}\right) $
satisfying (1) and (2). Then (3) follows from (2), the properties of
ultrafilter quantifiers, and the facts that, for $1\leq j\leq k$, $\mathcal{U%
}_{j}+\mathcal{U}_{k}=\mathcal{U}_{k}+\mathcal{U}_{j}=\mathcal{U}_{k}$ and $%
f\left( \mathcal{U}_{k}\right) =\mathcal{U}_{j}$ for any nondecreasing
surjection $f:k+1\rightarrow j+1$. This concludes the recursive
construction.\ In view of (1), the sequence $\left( b_{n}\right) $ obtained
through this construction has the property that $\mathrm{TS}_{j}\left(
b_{n}\right) $ is contained in $A_{j}$, and hence $c$ is constant on $%
\mathrm{TS}_{j}\left( b_{n}\right) $ for $j=1,2\ldots ,k$.
\end{proof}

\section{A multidimensional generalization\label{Section:multidimensional}}

Gowers' theorem on $\mathrm{FIN}_{k}$ can be seen as a generalization of
Hindman's theorem for sets of finite unions \cite{hindman_finite_1974}. Such
a theorem asserts that for any finite coloring of $\mathrm{FIN}_{1}$, there
exists a block sequence $\left( b_{n}\right) $ in $\mathrm{FIN}_{1}$ such
that $\mathrm{TS}_{1}\left( b_{n}\right) $ is monochromatic. Observe that
one can identify $\mathrm{FIN}_{1}$ with the set of nonempty finite subsets
of $\omega $. Then $\mathrm{TS}_{1}\left( b_{n}\right) $ is just the
collection of all finite unions of the elements of the given sequence.
Hindman's theorem on finite unions is the particular instance of Gowers'
theorem for $k=1$.

In another direction, Hindman's theorem on finite unions was generalized,
independently, by Milliken and Taylor \cite%
{milliken_ramseys_1975,taylor_canonical_1976}; see also \cite%
{bergelson_polynomial_2014}. Fix $m\in \mathbb{N}$ and consider the set $%
\mathrm{FIN}_{1}^{\left[ m\right] }$ of block sequences in $\mathrm{\mathrm{%
FIN}}_{1}$ of length $m$. The Milliken-Taylor theorem on finite unions
asserts that, for any finite coloring of $\mathrm{FIN}_{1}^{\left[ m\right]
} $, there exists an infinite block sequence $\left( b_{n}\right) $ in $%
\mathrm{FIN}_{1}$ such that the set $\mathrm{TS}_{1}\left( b_{n}\right) ^{%
\left[ m\right] }$ of $m$-tuples of the form%
\begin{equation*}
(b_{n_{0}}+\cdots +b_{n_{\ell _{0}-1}},b_{n_{\ell _{0}}}+\cdots +b_{n_{\ell
_{1}}-1},\ldots ,b_{n_{\ell _{m-1}}}+\cdots +b_{n_{\ell _{m}-1}})
\end{equation*}%
for $0<\ell _{0}<\ell _{2}<\cdots <\ell _{m}$ and $0\leq n_{1}<n_{2}<\cdots
<n_{\ell _{m}-1}$, is monochromatic.

The multidimensional analog of Gowers' theorem for a single tetris operation
is proved in \cite[Corollary 5.26]{todorcevic_introduction_2010}. The
corresponding finite version is considered in \cite{tyros_primitive_2015}.
In a similar spirit, one can consider a multidimensional generalization of
Theorem \ref{Theorem:multidimGowers}. Let $\mathrm{FIN}_{k}^{\left[ m\right]
}$ be the space of block sequences in $\mathrm{FIN}_{k}$ of length $m$, and $%
\mathrm{\mathrm{FI}N}_{\leq k}^{\left[ m\right] }$ be the union of $\mathrm{%
FIN}_{j}^{\left[ m\right] }$ for $j=1,2,\ldots ,k$. If $\left( b_{n}\right) $
is a block sequence in $\mathrm{FIN}_{k}$ and $1\leq j\leq k$, then we
define the \emph{tetris subspace} $\mathrm{TS}_{j}\left( b_{n}\right) ^{%
\left[ m\right] }$ of $\mathrm{FIN}_{k}^{\left[ m\right] }$ generated by $%
\left( b_{n}\right) $ to be the set of elements of $\mathrm{FIN}_{j}^{\left[
m\right] }$ of the form $(a_{0},\ldots ,a_{m-1})$, where $a_{d}$ for $d\in m$
is equal to%
\begin{equation*}
f_{n_{d}}\circ b_{n_{d}}+\cdots +f_{n_{d+1}-1}\circ b_{n_{d+1}-1}
\end{equation*}%
for some $n_{0}=0<n_{1}<n_{2}<\cdots <n_{m}$, $0\leq j_{i}\leq k$ and
nondecreasing surjections $f_{i}:k+1\rightarrow j_{i}+1$ for $i\in n_{m}$
such that $\max \left\{ j_{n_{d}},\ldots ,j_{n_{d+1}-1}\right\} =j$. We can
then state the multidimensional generalization of Theorem \ref%
{Theorem:multiGowers} as follows:

\begin{theorem}
\label{Theorem:multidimGowers}Suppose that $m,k\in \mathbb{N}$. For any
finite coloring of $\mathrm{FIN}_{\leq k}^{\left[ m\right] }$, there exists
an infinite block sequence $\left( b_{n}\right) $ in $\mathrm{FIN}_{k}$ such
that $\mathrm{TS}_{j}\left( b_{n}\right) ^{\left[ m\right] }$ is
monochromatic for every $j=1,2,\ldots ,k$.
\end{theorem}

In order to prove the Milliken-Taylor theorem, one can consider an
idempotent cofinite ultrafilter $\mathcal{U}_{1}$ on $\mathrm{FIN}_{1}$, and
then the Fubini power $\mathcal{V}_{1}:=\mathcal{U}_{1}^{\otimes m}$. This
is defined as the cofinite ultrafilter on $\mathrm{FIN}_{1}^{\left[ m\right]
}$ such that $A\in \mathcal{V}_{1}$ if and only if $\left( \mathcal{U}%
_{1}b_{1}\right) \cdots \left( \mathcal{U}_{m}b_{m}\right) $, $\left(
b_{1},\ldots ,b_{m}\right) \in A$; see \cite[\S 1.2]%
{todorcevic_introduction_2010}. Then any element of $\mathcal{V}$ witnesses
that the Milliken-Taylor theorem holds. A similar approach works for Theorem %
\ref{Theorem:multidimGowers}. Indeed, consider the cofinite ultrafilter $%
\mathcal{U}_{k}$ on $\mathrm{FIN}_{k}$ given by Lemma \ref{Lemma:idempotent}
and its Fubini power $\mathcal{V}_{k}:=\mathcal{U}_{k}^{\otimes m}$ on $%
\mathrm{FIN}_{k}^{\left[ m\right] }$. Then any element of $\mathcal{V}_{k}$
witness that Theorem \ref{Theorem:multidimGowers} holds. The proof of such a
fact is analogous to the proof of Theorem \ref{Theorem:multiGowers}, and
only notationally heavier. The details are left to the interested reader.

As usual, it follows by compactness from Theorem \ref{Theorem:multidimGowers}
the corresponding finite version, which recovers Corollary 2.3 of \cite%
{bartosova_gowers_2014}

\begin{corollary}
Suppose that $m,k,\ell ,r\in \mathbb{N}$. There exists $n\in \mathbb{N}$
such that for any $r$-coloring of $\mathrm{FIN}_{j}(n)^{\left[ m\right] }$,
there exists a block sequence $\left( b_{0},\ldots ,b_{\ell -1}\right) $ in $%
\mathrm{FIN}_{k}$ of length $\ell $ such that $\mathrm{TS}_{j}(b_{0},\ldots
,b_{\ell -1})^{\left[ m\right] }$ is monochromatic for $j=1,2,\ldots ,k$.
\end{corollary}

\section{A generalization for layered partial semigroups\label%
{Section:abstract}}

Recall that a \emph{partial semigroup}\textrm{\ }\cite[Definition 1.2]%
{farah_partition_2002} is a set $S$ endowed with a partially defined binary
operation $\left( x,y\right) \mapsto xy$ satisfying $\left( xy\right)
z=x\left( yz\right) $. This equation should be interpreted as asserting that
the left hand side is defined if and only if the right hand side is defined,
and in such a case the equality holds. Suppose that $x$ is an element of a
partial semigroup $S$. Following \cite[Definition 2.1]{farah_partition_2002}%
, we let $\varphi _{S}(x)$ be the set of elements $y$ of $S$ such that $xy$
is defined. More generally, for a subset $A$ of $S$, we let $\varphi
_{S}\left( A\right) $ be the set of elements $y$ of $S$ such that $xy$ is
defined for every $x\in A$. As in \cite[Definition 2.1]{farah_partition_2002}%
, we say that a partial semigroup $S$ is \emph{adequate}---also called \emph{%
directed }in \cite[\S 2.2]{todorcevic_introduction_2010}---if $\varphi
_{S}\left( A\right) $ is nonempty for every finite subset $A$ of $S$.

When $S$ is a directed partial semigroup, the set $\gamma S$ of ultrafilters 
$\mathcal{U}$ over $S$ with the property that $\left( \forall x\in S\right)
\left( \mathcal{U}y\right) $ $xy$ is defined, is a closed nonempty subset of
the space $\beta S$ of ultrafilters over $S$. One can define a compact right
topological semigroup operation on $\gamma S$ by setting $A\in \mathcal{UV}$
if and only if $\left( \mathcal{U}y\right) \left( \mathcal{V}z\right) $ $%
yz\in A$ \cite[Corollary 2.7]{todorcevic_introduction_2010}. More generally,
it is shown in \cite[Theorem 2.2]{farah_partition_2002} that the operation
on $S$ extends to a continuous map from $\beta S\times \gamma S$ to $\beta S$%
.

Suppose that $S,T$ are partial semigroups, and $\sigma :S\rightarrow T$ is a
function. We say that $\sigma $ is a \emph{partial semigroup homomorphism}
if for any $x,y\in S$, $\sigma (x)\sigma (y)$ is defined whenever $xy$ is
defined, and in such a case $\sigma (xy)=\sigma (x)\sigma (y)$ \cite[%
Definition 2.8]{farah_partition_2002}. We say that $\sigma :S\rightarrow T$
is an \emph{adequate partial semigroup homomorphism} if it is a partial
semigroup homomorphism with the property that for any finite subset $A$ of $%
S $ there exists a finite subset $B$ of $T$ such that $\varphi _{T}\left(
B\right) $ is contained in the image under $\sigma $ of $\varphi _{S}\left(
A\right) $.

If $T$ is a partial semigroup and $S\subset T$, then $S$ is an \emph{%
adequate partial subsemigroup }if the inclusion map $S\hookrightarrow T$ is
an adequate partial semigroup homomorphism \cite[Definition 2.10]%
{farah_partition_2002}. We say that a subset $S$ of a partial semigroup $T$
is an \emph{adequate\ ideal} if it is an adequate partial subsemigroup, and
for any $x\in S$ and $y\in T$ one has that $xy$ and $yx$ belong to $S$
whenever they are defined \cite[Definition 2.15]{farah_partition_2002}.
Lemma 2.14 and Lemma 2.16 of \cite{farah_partition_2002} show that, if $%
S\subset T$ is an adequate partial subsemigroup, then $\gamma S$ can be
canonically identified with a subsemigroup of $\gamma T$. If furthermore $S$
is an adequate ideal of $T$, then $\gamma S\ $is an ideal of $\gamma T$. We
now recall the definition of layered partial semigroup from \cite[\S 3]%
{farah_partition_2002}. An element $e$ of a partial semigroup is an \emph{%
identity element} if $ex$ and $xe$ are defined and equal to $x$ for any $%
x\in S$.

\begin{definition}
\label{Definition:layered}A \emph{layered partial semigroup} with $k$ layers
is a partial semigroup $S$ endowed with a partition $\left\{ S_{0},\ldots
,S_{k}\right\} $ such that $S_{0}=\left\{ e\right\} $ for some identity
element $e$ for $S$, and for every $n=1,2,\ldots ,k$, letting $S_{\leq
n}=S_{0}\cup \cdots \cup S_{k}$, one has that $S_{\leq n}$ is an adequate
partial semigroup, $S_{n}$ is an adequate partial subsemigroup of $S$, and
an adequate ideal of $S_{\leq n}$.
\end{definition}

In the following we will assume that $S$ is a layered partial semigroup with 
$k$ layers as witnessed by the partition $\left\{ S_{0},\ldots
,S_{k}\right\} $, and set $S_{\leq n}=S_{0}\cup \cdots \cup S_{n}$. Observe
that it follows from the definition of layered partial semigroup that $%
\gamma S_{n}$ is an ideal of $\gamma S_{\leq n}$, and a subsemigroup of $%
\gamma S$ for $n=1,2,\ldots ,k$.

\begin{definition}
\label{Definition:tetris_action}Suppose that $\mathcal{A}=\left( \mathcal{F}%
_{1},M_{1},\mathcal{F}_{2},M_{2},\ldots ,\mathcal{F}_{k},M_{k}\right) $ is a
tuple such that for every $n=1,2,\ldots ,k$, $\mathcal{F}_{n}$ is a nonempty
finite collection of partial semigroup homomorphisms from $S_{\leq n}$ to $%
S_{\leq n-1}$, and $M_{n}$ is an adequate subsemigroup of $S_{n}$ for $%
n=1,2,\ldots ,k$. We say that $\mathcal{A}$ is a \emph{tetris action} on $S$
if and only if it satisfies for any $n=2,3,\ldots ,k$, and $\sigma \in 
\mathcal{F}_{n}$ the following conditions:

\begin{enumerate}
\item the image of $M_{n}$ under $\sigma $ is an adequate partial
subsemigroup of $M_{n-1}$;

\item the image of $S_{n}$ under $\sigma $ is an adequate partial
subsemigroup of $S_{n-1}$;

\item the restriction of $\sigma $ to $S_{\leq n-1}$ either belongs to $%
\mathcal{F}_{n-1}$, or it is the identity map of $S_{\leq n-1}$, and

\item for any $\sigma _{1},\sigma _{2}\in \mathcal{F}_{n}$ one has that $%
\sigma _{1}|_{M_{n}}=\sigma _{2}|_{M_{n}}$.
\end{enumerate}
\end{definition}

From now on we assume that $\left( \mathcal{F}_{1},M_{1},\mathcal{F}%
_{2},M_{2},\ldots ,\mathcal{F}_{k},M_{k}\right) $ is a tetris action on $S$
as in Definition \ref{Definition:tetris_action}. It follows from \cite[Lemma
2.4]{farah_partition_2002} that for any $n=2,3,\ldots ,n$, any element $%
\sigma $ of $\mathcal{F}_{n}$ admits a continuous extension $\sigma :\beta
S_{\leq n}\rightarrow \beta S_{\leq n-1}$ such that:

\begin{itemize}
\item if $p\in \beta S_{n}$, $q\in \gamma S_{\leq n-1}$, and $\sigma (q)\in
\gamma S_{\leq n-1}$, then $\sigma (pq)=\sigma (p)\sigma (q)$;

\item if $p\in \beta S_{\leq n-1}$, $q\in \gamma S_{n}$, and $\sigma (q)\in
\gamma S_{n}$, then $\sigma (pq)=\sigma (p)\sigma (q)$;

\item $\sigma $ maps $\gamma S_{n}$ to $\gamma S_{n-1}$ and $\gamma M_{n}$
to $\gamma M_{n-1}$.
\end{itemize}

In particular, $\sigma $ induces continuous semigroup homomorphism $\sigma
:\gamma S_{n}\rightarrow \gamma S_{n-1}$ mapping the subsemigroup $\gamma
M_{n}$ to $\gamma M_{n-1}$. The same proof as Lemma \ref{Lemma:idempotent}
shows the following:

\begin{lemma}
\label{Lemma:idempotent_abstract}There exist idempotent elements $\mathcal{U}%
_{n}\in \gamma S_{n}$ for $n=1,2,\ldots ,k$ such that $\sigma \left( 
\mathcal{U}_{n}\right) =\mathcal{U}_{n-1}$ and $\mathcal{U}_{n}\mathcal{U}%
_{n-1}=\mathcal{U}_{n-1}\mathcal{U}_{n}=\mathcal{U}_{n}$ for every $%
n=2,\ldots ,k$ and $\sigma \in \mathcal{F}_{n}$.
\end{lemma}

Given a tetris action, one can define as in \cite[Definition 3.9]%
{farah_partition_2002} the collection $\mathcal{G}_{n}$ of maps from $S_{k}$
to $S_{n}$ of the form $\sigma _{n+1}\circ \sigma _{n+2}\circ \cdots \circ
\sigma _{k}$, where $\sigma _{j}\in \mathcal{F}_{j}$ for $j=n+1,\ldots ,k$.
We also let $\mathcal{G}$ be the union of $\mathcal{G}_{n}$ for $%
n=1,2,\ldots ,k$.

\begin{definition}
\label{Definition:block_sequence}A\emph{\ block sequence} in $S_{k}$ is a
sequence $\left( b_{n}\right) $ such that $f_{0}\left( b_{0}\right) \cdots
f_{n}\left( b_{n}\right) $ is defined for any $n\in \omega $ and $%
f_{0},\ldots ,f_{n}\in \mathcal{G}$.
\end{definition}

The notion of block sequence in $S_{n}$ for some $n\leq k$ is defined
similarly. We let $S_{n}^{\left[ m\right] }$ be the set of block sequences
in $S_{n}$ of length $m$, and $S_{\leq n}^{\left[ m\right] }$ be the union
of $S_{j}^{\left[ m\right] }$ for $j=1,2,\ldots ,n$. If $\left( b_{n}\right) 
$ is a block sequence in $S_{k}$, then we define the \emph{tetris subspace} $%
\mathrm{TS}_{j}\left( b_{n}\right) \subset S_{j}^{\left[ n\right] }$ of the $%
j$-th layer generated by $\left( b_{n}\right) $ to be the set of elements of 
$S_{j}^{\left[ n\right] }$ of the form $\left( a_{0},\ldots ,a_{m-1}\right) $
where for some $0=n_{0}<n_{1}<\cdots <n_{m}\in \omega $, $j_{i}\in k+1$, and 
$f_{i}\in \mathcal{G}_{j_{i}}$ for $i\in n_{m}$ one has that for every $d\in
m$, $\max \{j_{n_{d}},\ldots ,j_{n_{d+1}-1}\}=j$ and $%
a_{d}=f_{n_{d}}(b_{n_{d}})\cdots f_{n_{d+1}-1}(b_{n_{d+1}-1})$. Then using
Lemma \ref{Lemma:idempotent_abstract} one can prove as in Theorem \ref%
{Theorem:multidimGowers} the following result, which is a common
generalization of the Galvin-Glazer theorem and Theorem \ref%
{Theorem:multidimGowers}.

\begin{theorem}
\label{Theorem:abstract}Suppose that $S$ is a layered partial semigroup
endowed with a tetris action as above. Fix $m\in \mathbb{N}$ and a finite
coloring of $S_{\leq k}^{\left[ m\right] }$. Then there exists an infinite
block sequence $\left( b_{n}\right) $ in $S_{k}$ such $\mathrm{TS}_{j}\left(
b_{n}\right) ^{\left[ m\right] }$ is monochromatic for every $j=1,2,\ldots
,k $.
\end{theorem}

It is clear that Theorem \ref{Theorem:abstract} has the Galvin-Glazer
theorem \cite[Theorem 2.20]{todorcevic_introduction_2010} as a particular
case. Set now $S_{j}:=\mathrm{FIN}_{j}$ for $j=0,1,\ldots ,k$, and $%
S:=S_{0}\cup \cdots \cup S_{n}$. Define a partial semigroup operation on $S$
by $\left( b,b^{\prime }\right) \mapsto b+b^{\prime }$ whenever $\mathrm{Supp%
}\left( b\right) <\mathrm{\mathrm{Sup}p}\left( b^{\prime }\right) $, where $%
b+b^{\prime }$ is the pointwise sum. Then $S=S_{0}\cup \cdots \cup S_{k}%
\mathrm{\ }$is a layered partial semigroup in the sense of Definition \ref%
{Definition:layered}. Denote by $\mathcal{F}_{n}$ for $n=1,2,\ldots ,k$ the
collection of multiple tetris operations $T_{1},\ldots ,T_{n}:\mathrm{FIN}%
_{n}\rightarrow \mathrm{FIN}_{n-1}$ defined in the introduction. Let also $%
M_{n}\subset S_{n}$ be the set of $b\in S_{n}$ such that $b\left( i\right)
\in \left\{ 0,n\right\} $ for every $i\in \omega $. It is then easy to see
that $\left( \mathcal{F}_{n},M_{n}\right) _{n=1}^{k}$ is a tetris action on $%
S$ in the sense of Definition \ref{Definition:tetris_action}. Furthermore
the conclusions of Theorem \ref{Theorem:abstract} in the particular case of
such a tetris action yields Theorem \ref{Theorem:multidimGowers}.

\bibliographystyle{amsplain}
\bibliography{bibliography}

\end{document}